\documentclass[11pt,a4paper]{article}
\usepackage{amsmath,amsfonts,amssymb,mathrsfs}
\usepackage{graphicx}
\usepackage{subfigure}
\usepackage{pstricks}
\usepackage{pst-node}
\usepackage{xcolor}
\newrgbcolor{blu}{.1 0 .7}
\newrgbcolor{lgrey}{.9 .9 .9}
\newrgbcolor{grey}{.7 .7 .7}
\newrgbcolor{dgrey}{.3 .3 .3}
\newrgbcolor{lred}{1 .8 .8}
\newrgbcolor{mlred}{1 .7 .7}
\newrgbcolor{bblue}{.6 .6 1}
\newrgbcolor{lblue}{.8 .8 1}
\newrgbcolor{mblue}{.2 .25 1}
\newrgbcolor{dblue}{.3 .32 1}
\newrgbcolor{ddblue}{.4 .4 1}
\newrgbcolor{lgreen}{.8 1 .8}
\newrgbcolor{mgreen}{ .7 1 .7}
\newrgbcolor{ggreen}{ .1 .75 .35}
\newtheorem{theorem}{\bf Theorem}[section]
\newtheorem{lemma}[theorem]{\bf Lemma}

\newtheorem{definition}[theorem]{\bf Definition}
\newtheorem{remark}[theorem]{\bf Remark}
\newtheorem{proposition}[theorem]{\bf Proposition}

\newcommand{\R}{\mathbb{R}}

\def \rn {{\mathbb {R}}^{N}}
\def \rnn {{\mathbb {R}}^{2n+1}}
\def \rNN {{\mathbb {R}}^{N+1}}

\def \K  {{K_0}}  
\def \L {\mathscr{L}}

\def \A {\mathscr{A}}

\def\p{\partial}

\newrgbcolor{michelared}{ .8 0 0}
\newenvironment{michelarev}{\color{michelared}}{\color{black}}
\newcommand{\bmicr}{\begin{michelarev}}
\newcommand{\emicr}{\end{michelarev}}

\newrgbcolor{sergioblue}{ .02 0 .7}
\newenvironment{sergiorev}{(\color{sergioblue}}{\color{black})}
\newcommand{\bsr}{\begin{sergiorev}}
\newcommand{\esr}{\end{sergiorev}}

\newrgbcolor{fragreen}{ 0 .51 0}
\newenvironment{frarev}{(\color{fragreen}}{\color{black})}
\newcommand{\bfra}{\begin{frarev}}
\newcommand{\efra}{\end{frarev}}

\newcommand{\denoterow}[1]{\rlap{\hspace{1em}$\leftarrow{}$ $j^{th}$}}

\def \D {{\Delta}}
\def \LL {{\Lambda}}
\def \AS {{ \mathscr{A}_{(v_{0}, x_{0}, t_{0})} }}
\def \ASY {{ \mathscr{A}_{(y_{0}, t_{0})} }}

\newenvironment{proof}{\noindent {\sc Proof.}}{\hfill $\square$}

\def \o {{\omega}}

\def \g {{\gamma}}
\def \d {{\delta}}

\def \epsilon {{\varepsilon}}

\def \l {{\lambda}}

\def \t {{\tau}}

\def \y {{\eta}}

\def \phi {{\varphi}}

\def \O {{\Omega}}

\usepackage[english]{babel}
\begin{document}
\title{A geometric statement of the Harnack inequality for a degenerate Kolmogorov equation with rough coefficients}
\author{{\sc{Francesca Anceschi}
\thanks{Dipartimento di Scienze Fisiche, Informatiche e Matematiche, Universit\`{a} di Modena e Reggio Emilia, Via
Campi 213/b, 41125 Modena (Italy). E-mail: francesca.anceschi@unimore.it}
{\sc{Michela Eleuteri}
\thanks{Dipartimento di Scienze Fisiche, Informatiche e Matematiche, Universit\`{a} di Modena e Reggio Emilia, Via
Campi 213/b, 41125 Modena (Italy). E-mail: michela.eleuteri@unimore.it}
\sc{Sergio Polidoro}
\thanks{Dipartimento di Scienze Fisiche, Informatiche e Matematiche, Universit\`{a} di 
Modena e Reggio Emilia, Via Campi 213/b, 41125 Modena (Italy). E-mail: 
sergio.polidoro@unimore.it}
}}}
\maketitle

\bigskip

\begin{abstract}
We consider weak solutions of second order partial differential equations of 
Kolmogorov-Fokker-Planck type with measurable coefficients in the form
$$
   \partial_t u + \langle v, \nabla_x u \rangle  ={\rm div}_v (A(v,x,t) \nabla_v u) + 
   \langle b(v,x,t), \nabla_v u \rangle  + f, \quad (v,x,t) \in \mathbb{R}^{2n+1},
$$
where $A$ is an uniformly positive symmetric matrix with bounded measurable 
coefficients, $f$ and the components of the vector $b$ are bounded and measurable 
functions.  We give a geometric statement of the Harnack inequality recently proven by 
Golse, Imbert, Mouhot and Vasseur. As a corollary we obtain a strong maximum principle. 
\end{abstract}

\normalsize

\section{Introduction}
We consider second order partial differential equations of Kolmogorov-Fokker-Planck type in the form
\begin{equation} \label{eq-4-autori-detailed}
\begin{split} 
  \partial_t u(v,x,t) + \sum_{j=1}^{n} v_{j} \partial_{x_j} u(v,x,t)  = &
  \sum_{j,k=1}^{n} \partial_{v_j} \left( a_{jk}(v,x,t) \partial_{v_k} u(v,x,t) \right) \\
 + & \, \sum_{j=1}^{n} b_{j}(v,x,t) \partial_{v_j} u(v,x,t) +  f(v,x,t), \qquad (v,x,t) \in \Omega, 
\end{split}
\end{equation}
where:
\begin{description}
  \item [{\it i)}] $\Omega$ is an open subset of $\mathbb{R}^{2n+1}$; 
  \item [{\it ii)}] $A = \left(a_{jk}\right)_{j,k=1, \dots, n}$ is a symmetric matrix with real measurable entries. 
Moreover,  
there exist two positive constants $\lambda, \Lambda$ such that
\begin{equation*}
  \lambda | \xi |^2 \le \langle A(v,x,t) \xi, \xi \rangle \le \Lambda | \xi |^2, \qquad \forall (v,x,t) \in \Omega, 
\quad \forall \xi \in \mathbb{R}^n;
\end{equation*}
  \item  [{\it iii)}] $b = \left( b_1, \dots, b_n \right)$ is a vector of $\mathbb{R}^n$ with bounded measurable 
coefficients; 
  \item  [{\it iv)}] $f \in L^{\infty}(\Omega)$.
\end{description}

As the coefficients of the operator $\L$ are measurable, we need to consider \emph{weak solutions} of $\L u = f$ in the 
following sense. Consider any open subset $\O$ of $\rnn$. A weak solution to \eqref{eq-4-autori-detailed} is a function 
$u \in L^{2}_{\text{loc}}(\O)$ such that $\partial_{v_1} u, \dots, \partial_{v_n} u$ and the directional derivative  
$\partial_t u  + \langle v, \nabla_x u \rangle$ belong to $L^{2}_{\text{loc}}(\O)$, and moreover
\begin{equation*} 
   \int_{\O} (\partial_t u  + \langle v, \nabla_x u  \rangle - \langle 
b, \nabla_v u \rangle) \phi  dv \, dx \, dt + 
   \int_{\O}  \langle A \nabla_v u, \nabla_v \phi \rangle dv \, dx \, dt 
   = \int_{\O} f \phi  dv \, dx \, dt,
\end{equation*}
for every $\forall \phi\in C_{0}^{\infty}(\O)$. In the sequel of this note the equation \eqref{eq-4-autori-detailed} 
will be understood in the weak sense and will be written in the short form, $\L u = f$, where
\begin{equation} \label{eq-4-autori} 
\L u =: \partial_t u  + \langle v, \nabla_x u \rangle - 
{\rm div}_v (A \nabla_v u) - \langle b, \nabla_v u \rangle,
\qquad (v,x,t) \in \Omega.
\end{equation}

\medskip

The motivation for studying equation \eqref{eq-4-autori} comes from the Stochastic theory and from its applications to 
several research fields. Indeed, the operator $\mathscr{L}_0$ defined as
\begin{equation}
\label{eq-Kolmo}
\mathscr{L}_0 u := \partial_t u + \langle v, \nabla_x u \rangle - \tfrac12 \, {\rm div}_v \left( \nabla_v u \right),
\end{equation}
was considered by Kolmogorov in \cite{Kolmogorov} to describe the probability density of a system with  $2n$ degrees of 
freedom. Precisely, the fundamental solution $\Gamma = \Gamma(v,x,t; v_0, x_0, t_0)$ of \eqref{eq-Kolmo} is the density 
of the stochastic process
\begin{equation}
\label{eq-langevin}
\begin{cases}
  V_t = v_0 +  W_{t-t_0}, \\
  X_t = x_0 + \int_{t_0}^t V_s ds, \qquad t > t_0,
\end{cases}
\end{equation}
which is a solution to the Langevin equation $dV_t = dW_t, dX_t = V_t dt$. Here $(W_t)_{t>0}$ denotes a Wiener 
process. Note that $\mathscr{L}_0$ is a particular case of the differential operator appearing in \eqref{eq-4-autori}, 
as we choose $A$ equal to the $n \times n$ identity matrix $I_n$, multiplied by $\frac12$, and $b = 0$.

Other applications of the operator in \eqref{eq-Kolmo} arise in the kinetic theory of gases. In this setting $\L$ takes 
the following general form 
\begin{equation}
\label{eq-kinetic}
  Y u = \mathcal{Q} [u],
\end{equation}
where $Y$ denotes the \emph{total derivative} with respect to the time in the phase space $(v,x) \in \mathbb{R}^n\times 
\mathbb{R}^n$
\begin{equation}
Y u :=  \partial_t u + \langle v, \nabla_x u \rangle,
\end{equation}
while $\mathcal{Q}$ is a collision operator, which can occur in the form of a second order differential operator, 
acting on the velocity variable $v$, that can appear either in linear or in nonlinear form.  In the 
Fokker-Planck-Landau 
model $\mathcal{Q}$ depends on the variable $v$ and on the unknown solution $u$ through some integral expressions. For 
the description of the stochastic processes and kinetic models leading to equations of the type \eqref{eq-4-autori}, we 
refer to the classical monographs \cite{Chandresekhar}, \cite{DudMart} and \cite{ChapCow}. 

We also mention that equations similar to \eqref{eq-4-autori} appear in Finance. For instance the equation
\begin{equation}
  \partial_t + \tfrac{1}{2} \sigma^2 S^2 \partial_S^2 V + S \partial_A V + r ( S \partial_S V- V) = 0 
\end{equation}
occurs  in the Black \& Scholes framework when considering the problem of pricing Asian options. We refer to 
\cite{Alziary, BarraquandPudet, BarucciPolidoroVespri}, and \cite{Pascucci-book} for a more detailed discussion of 
this topic. 

\medskip

The main theoretical interest in the operator $\mathscr{L}_0$ relies on its regularity properties, first noticed by 
Kolmogorov. Indeed, Kolmogorov writes in \cite{Kolmogorov} the explicit expression of the fundamental solution $\Gamma$ 
of \eqref{eq-Kolmo}, and points out the remarkable fact that it is a $C^\infty$ smooth function, despite the strong 
degeneracy of its characteristic form. Later, H\"ormander in \cite{Hormander} considers $\mathscr{L}_0$ as the 
prototype of a wide family of \emph{degenerate hypoelliptic operators}, with the following meaning.  

\emph{Let $\Omega$ be an open subset of $\R^{2n+1}$. We say that $\mathscr{L}_0$ is {\sc hypoelliptic} in $\O$ if, for 
every measurable function $u: \Omega \to \R$ which solves the equation $\L_0 u = f$ in the distributional sense, we 
have} 
\begin{equation} \label{eq-hypoellipticity}
  f \in C^\infty(\Omega) \quad \Rightarrow \quad u \in C^\infty(\Omega).
\end{equation}
In Section 2 we recall some known results about $\L_0$ and about more general linear second order differential 
operators, that in the sequel will be denoted by $\K$ (see \eqref{eq-Kolmo-RN} below), satisfying the  H\"ormander's 
hypoellipticity condition introduced in \cite{Hormander}. Since the works by Folland \cite{Folland}, Rotschild and 
Stein \cite{RothschildStein}, Nagel, Stein and Wainger \cite{NagelSteinWainger} concerning operators satisfying the 
H\"ormarder's condition, it is known that the natural framework for the regularity of that operators is the analysis on 
Lie groups. 
The first study of the \emph{non-Euclidean} translation group related to the degenerate Kolmogorov operators 
$\K$ has been performed by Lanconelli and one of the authors in \cite{LP94}. This non-commutative structure underlying 
$\L_0$ has replaced the usual \emph{Euclidean translations} and the \emph{parabolic dilations} in the study of 
operators 
$\L$ with variable coefficients $a_{jk}$'s and $b_j$'s. The development of the regularity theory for operators $\L$ has 
been achieved in several steps, paralleling the history of the uniformly parabolic equations. 

In particular, several interesting results have been obtained as the definition of {\it H\"older continuous functions} 
is given in terms of the Lie group relevant to $\L_0$.  We refer to Weber \cite{Weber}, Il'in \cite{Il'in}, Eidelman et 
al. \cite{Eidelman-et-al}, Polidoro \cite{Polidoro2}, Delarue and Menozzi \cite{DelarueMenozzi} for the construction of 
a fundamental solution based on the \emph{parametrix method}. We quote \cite{Polidoro2}, \cite{Polidoro1} 
for the proof of the upper and lower bounds for the fundamental solution, of mean value formulas and Harnack 
inequalities for the non-negative solutions $u$ of $\L u = 0$. Schauder type estimates have been proved by Satyro  
\cite{Satyro}, Lunardi \cite{Lunardi}, Manfredini \cite{Manfredini}. Analogous results have been proven in a more 
general context by Morbidelli \cite{Morbidelli}, Di Francesco and Pascucci \cite{DiFrancescoPascucci}, and Di Francesco 
and Polidoro \cite{DiFrancescoPolidoro}.  

The study of  the operator $\L$ with {\it measurable coefficients} has required some tools for the construction of a 
functional analysis on the Lie group relevant to $\L_0$. In the work by Pascucci and Polidoro \cite{PP04}, the 
classical iterative method introduced by Moser (\cite{Moser3}, \cite{Moser3bis}), which in turn relies on the 
combination of a Caccioppoli inequality with a Sobolev inequality, have been used to obtain a $L^\infty$ 
upper bound for the weak solutions of $\L u = 0$. The Sobolev inequality has been obtained in \cite{PP04} by using the 
fundamental solution $\Gamma$ of $\L_0$ and its invariance with respect to the Lie group related to $\L_0$. The 
methods and the results of \cite{PP04} have been then extended to Kolmogorov type operators on non-homogeneous Lie 
groups by Cinti, Pascucci and Polidoro \cite{CintiPascucciPolidoro}; we also recall \cite{CintiPolidoro} and  
\cite{LP17} where similar techniques have been adapted to the non-Euclidean setting to prove $L^\infty$ local estimates 
for the solutions.

A further important step in the functional analysis for operators $\L$ and for its regularity theory has been done by  
Wang and Zhang \cite{WangZhang1, WangZhang2}, who have proven a weak form of the Poincar\'e inequality and the 
H\"older continuity of the weak solutions $u$ of $\L u = 0$. More recently, Golse, Imbert, Mouhot and Vasseur 
\cite{4authors} provide us with an alternative proof of the  H\"older continuity of the solutions and prove an 
invariant 
Harnack inequality for the positive solutions of $\L u = 0$. Based on the Harnack inequality of \cite{4authors}, 
Lanconelli, Pascucci and Polidoro prove in \cite{LPP} Gaussian upper and lower bounds for the fundamental solution of 
$\L$ (see also \cite{LP17}). 

\medskip

In this note we prove a geometric version of the Harnack inequality proved in \cite{4authors}, whose statement is 
recalled in Theorem \ref{harnack} below, after some preliminary notation. In the unit box  of $\rnn$:
\begin{equation} \label{eq-unitbox} 
	Q \hspace{1mm} = \hspace{1mm} ]-1,1[^n \times ]-1,1[^n \times ]-1,0[,
\end{equation}
it reads as the usual parabolic Harnack inequality: there exist two \emph{small} boxes $Q^+$ and $Q^-$ contained in 
$Q$, 
with $Q^+$ located \emph{above} $Q^-$ with respect to the time variable, and a positive constant $M$, such that 
\[ 
		\sup_{Q^{-}} u \le M (\inf_{Q^{+}} u + \left\| f \right\|_{L^{\infty} (Q)})
\]
for every non-negative solution $u$ of $\L u = f$ in $Q$, with $f \in L^\infty(Q)$.  


\begin{center}
 \begin{pspicture*}(-8,-3)(8,3)
\pcline[linecolor=dgrey](-3,0)(1,-1)%
\pcline[linecolor=dgrey](-3,-2)(1,-3)%
\pcline[linecolor=dgrey](3,0)(-1,1)%
\pcline[linecolor=dgrey, linestyle=dashed](3,-2)(-1,-1)%
\pcline[linecolor=dgrey](3,0)(3,-2)%
\pcline[linecolor=dgrey](-3,0)(-3,-2)%
\pcline[linecolor=dgrey](1,-1)(1,-3)%
\pcline[linecolor=dgrey, linestyle=dashed](-1,1)(-1,-1)%
\pcline[linecolor=dgrey](3,0)(1,-1)%
\pcline[linecolor=dgrey](-3,0)(-1,1)%
\pcline[linecolor=dgrey](3,-2)(1,-3)%
\pcline[linecolor=dgrey, linestyle=dashed](-1,-1)(-3,-2)
\pcline[linecolor=dgrey](0,-2)(0,1.4)
\pcline[linecolor=dgrey](0,1.4)(-.05,1.2)%
\pcline[linecolor=dgrey](0,1.4)(.05,1.2)%
\uput[0](0,1.4){$t$}
\pspolygon[fillstyle=solid,fillcolor=lgrey](0.8,-0.12)(0.6,-0.22)(-0.8,0.12)(-0.6,0.22)
\pspolygon[fillstyle=solid,fillcolor=lgrey](-0.8,0.12)(-0.8,-0.43)(0.6,-0.77)(0.6,-0.22)
\pspolygon[fillstyle=solid,fillcolor=lgrey](0.8,-0.12)(0.6,-0.22)(0.6,-0.77)(0.8,-0.67)
\pspolygon[fillstyle=solid,fillcolor=lgrey](0.8,-1.12)(0.6,-1.22)(-0.8,-0.88)(-0.6,-0.78)
\pspolygon[fillstyle=solid,fillcolor=lgrey](-0.8,-0.88)(-0.8,-1.43)(0.6,-1.77)(0.6,-1.22)
\pspolygon[fillstyle=solid,fillcolor=lgrey](0.8,-1.12)(0.6,-1.22)(0.6,-1.77)(0.8,-1.67)

\pcline[linecolor=dgrey,linestyle=dashed](0,-2)(0,1.4)
\pcline[linecolor=dgrey](-4,1)(4,-1)%
\pcline[linecolor=dgrey](4,-1)(3.8,-1)%
\pcline[linecolor=dgrey](4,-1)(3.9,-.9)%
\uput[0](4,-1){$v$}
\pcline[linecolor=dgrey](2,1)(-2,-1)%
\pcline[linecolor=dgrey](-2,-1)(-1.8,-1)%
\pcline[linecolor=dgrey](-2,-1)(-1.9,-.9)%
\uput[180](-2,-1){$x$}
\uput[90](-1.8,0.8){$Q$}
\uput[60](0.8,-0.25){$Q^+$}
\uput[300](-1,-1.5){$Q^-$}
\end{pspicture*}
\end{center}
\begin{center}
  {\scriptsize \sc Fig. 1 - Harnack inequality.}
\end{center}

\medskip

We recall that, in the classical statement of the Harnack inequality for uniformly parabolic operators with measurable 
coefficients, the size of the boxes $Q^+$ and $Q^-$, and the gap between the lower basis of $Q^+$ and upper basis 
of $Q^-$ can be arbitrarily chosen (see Theorem of \cite{Moser3}). On the contrary, in the statement of the Harnack 
inequality for the operator $\L$ given in \cite{4authors}, neither the size of the boxes $Q^+$ and $Q^-$, nor their 
\emph{position} in $Q$ is characterized. Actually, as we will see in the sequel, it is known that the Harnack 
inequality does not hold for any choice of the boxes $Q^+$ and $Q^-$. This fact was previously noticed by Cinti, 
Nystr\"om and Polidoro in \cite{CNP10}, where classical solutions of $\L_0 u = 0$ are considered, and by Kogoj and 
Polidoro in \cite{KogojPolidoro}. We give here a sufficient condition for the validity of the Harnack inequality. For 
its precise statement we refer to the notion of \emph{attainable set} $\AS$ given in Definition \ref{prop-set} below. 
In the sequel ${\rm int}\big(\AS \big)$ denotes the interior of $\AS$. 

\begin{theorem} \label{mainthe}
Let $\O$ be an open subset of $\rnn$ and let $f \in L^{\infty} (\O)$. For every $(v_{0}, x_{0}, t_{0}) \in \O$, and for 
any compact set $K \subseteq {\rm int} \big(\AS \big)$, there exists a positive constant $C_{K}$, 
only dependent on $\O$, $(v_{0}, x_{0}, t_{0})$, $K$ and on the operator $\L$, such that 
\begin{equation*}
  \sup_{K} u \le C_{K} \left( u(v_{0}, x_{0}, t_{0}) + \|f\|_{L^{\infty}(\O)} \right),
\end{equation*}
for every non-negative solution to $\L u = f$.
\end{theorem}

We note that any weak solution $u$ of $\L u = f$ is H\"older continuous (see \cite{WangZhang1, WangZhang2} for the 
equation $\L u = 0$, and Theorem 2 in \cite{4authors} for  $\L u = f$ with $f \in L^\infty$), then $u(v_{0}, x_{0}, 
t_{0})$ is well defined. As we will see in the Definition \ref{prop-set}, the attainable set $\AS$ depends on the 
geometry of $\O$, and it can be easily described. For instance, when it agrees with the unit box $Q$ in 
\eqref{eq-unitbox} we have that 
\begin{equation} \label{eq-prop-set} 
	\A_{(0,0,0)} = \Big\{ (v,x,t) \in Q \mid |x_j| \le |t|, j=1, \dots, n \Big\}.
\end{equation} 
The proof of this fact can be seen in \cite{CNP10}, Proposition 4.5, p.$353$.


\begin{center}
 \begin{pspicture*}(-8,-3)(8,3)
\pcline[linecolor=dgrey](0,-2)(0,1.4)
\pcline[linecolor=dgrey](0,1.4)(-.05,1.2)%
\pcline[linecolor=dgrey](0,1.4)(.05,1.2)%
\uput[0](0,1.4){$t$}
\pcline[linecolor=dgrey](-4,1)(4,-1)%
\pcline[linecolor=dgrey](4,-1)(3.8,-1)%
\pcline[linecolor=dgrey](4,-1)(3.9,-.9)%
\uput[0](4,-1){$v$}
\pcline[linecolor=dgrey](2,1)(-2,-1)%
\pcline[linecolor=dgrey](-2,-1)(-1.8,-1)%
\pcline[linecolor=dgrey](-2,-1)(-1.9,-.9)%
\uput[180](-2,-1){$x$}
\pcline[linecolor=dgrey](-3,0)(1,-1)%
\pcline[linecolor=dgrey](-3,-2)(1,-3)%
\pcline[linecolor=dgrey](3,0)(-1,1)%
\pcline[linecolor=dgrey, linestyle=dashed](3,-2)(-1,-1)%
\pcline[linecolor=dgrey](3,0)(3,-2)%
\pcline[linecolor=dgrey](-3,0)(-3,-2)%
\pcline[linecolor=dgrey](1,-1)(1,-3)%
\pcline[linecolor=dgrey, linestyle=dashed](-1,1)(-1,-1)%
\pcline[linecolor=dgrey](3,0)(1,-1)%
\pcline[linecolor=dgrey](-3,0)(-1,1)%
\pcline[linecolor=dgrey](3,-2)(1,-3)%
\pcline[linecolor=dgrey, linestyle=dashed](-3,-2)(-1,-1)%
\pcline[linecolor=dgrey](-2,.5)(2,-.5)%
\pcline[linecolor=dgrey, linestyle=dashed](0,0)(0,-2)%
\pcline[linecolor=dgrey](2,-.5)(1,-3)%
\pcline[linecolor=dgrey, linestyle=dashed](-2,.5)(-1,-1)%
\pspolygon[fillstyle=solid,fillcolor=lgrey](-2,.5)(2,-.5)(1,-3)(-3,-2)
\pspolygon[fillstyle=solid,fillcolor=grey](2,-.5)(3,-2)(1,-3)
\pcline[linecolor=dgrey](2,1)(-2,-1)%
\pcline[linecolor=dgrey](-2,-1)(-1.8,-1)%
\pcline[linecolor=dgrey](-2,-1)(-1.9,-.9)%
\uput[180](-2,-1){$x$}
\uput[90](-1.8,0.8){$Q$}
\pcline[linecolor=dgrey](-3,0)(1,-1)%
\pcline[linecolor=dgrey](-3,-2)(1,-3)%
\pcline[linecolor=dgrey](3,0)(-1,1)%
\pcline[linecolor=dgrey, linestyle=dashed](3,-2)(-1,-1)%
\pcline[linecolor=dgrey](3,0)(3,-2)%
\pcline[linecolor=dgrey](-3,0)(-3,-2)%
\pcline[linecolor=dgrey](1,-1)(1,-3)%
\pcline[linecolor=dgrey, linestyle=dashed](-1,1)(-1,-1)%
\pcline[linecolor=dgrey](3,0)(1,-1)%
\pcline[linecolor=dgrey](-3,0)(-1,1)%
\pcline[linecolor=dgrey](3,-2)(1,-3)%
\pcline[linecolor=dgrey, linestyle=dashed](-3,-2)(-1,-1)%
\pcline[linecolor=dgrey](-2,.5)(2,-.5)%
\pcline[linecolor=dgrey, linestyle=dashed](0,0)(0,-2)%
\pcline[linecolor=dgrey](2,-.5)(1,-3)%
\pcline[linecolor=dgrey, linestyle=dashed](-2,.5)(-1,-1)%
\dotnode(0,0,0){O}
\uput[115](0.16,0){$(0,0,0)$}
\end{pspicture*}
\end{center}
\begin{center}
  {\scriptsize \sc Fig. 2 - $\mathscr{A}_{(0,0,0)}(Q)$.}
\end{center}

\medskip

A direct consequence of our main result inequality is the following strong maximum principle. 

\begin{theorem} \label{maincor}
Let $\O$ be an open subset of $\rnn$,  and let $u$ be a non-negative solution to $\L u = 0$.
If $u(v_{0}, x_{0}, t_{0})=0$ for some $(v_{0}, x_{0}, t_{0}) \in \O$, then $u(v,x,t) = 0$ for every  $(v,x,t) \in 
\overline{\AS}$.
\end{theorem}

Note that the Theorem \ref{maincor} extends to weak solution to $\L u = 0$ the well known Bony's strong maximum 
principle \cite{Bony69} for classical solutions of degenerate hypoelliptic Partial Differential Equations with smooth 
coefficients. We also recall the work of Amano \cite{Amano79}, where differential operators with continuous 
coefficients are considered.

We also note that the Theorem \ref{maincor} is somehow optimal. Indeed, in Proposition 4.5 of \cite{CNP10} it is 
shown that there exists a non-negative solution $u$ to $\L_0 u = 0$ in $Q$ such that $u(v,x,t) = 0$ for every  $(v,x,t) 
\in \overline{\A_{(0,0,0)}}$, and $u(v,x,t) > 0$ for every  $(v,x,t) \in Q \backslash \overline{\A_{(0,0,0)}}$.

\medskip

This article is organized as follows. Section 2 contains some preliminary results and known facts about the 
regularity properties of the operator $\L_0$ and on its invariance with respect to a non-Euclidean group structure on 
$\rnn$. It also contains a short discussion of the controllability problem related to $\L_0$ and the Definition of the 
Attainable set. In Section 3 we recall the Harnack inequality given in \cite{4authors} and we prove a 
dilation-invariant version of it. In Section 4 we prove our main results. 

\section{Preliminaries}
\setcounter{equation}{0}
In this Section we recall some known facts on the equation \eqref{eq-4-autori}, and on its prototype \eqref{eq-Kolmo}, 
that will play an important role in our study. We first recall that \eqref{eq-Kolmo} belongs to the more general class 
of differential operators considered in \cite{LP94}. Specifically, in \cite{LP94} have been studied operators in the 
following form
\begin{equation} \label{eq-Kolmo-RN}
    \K u : = \sum_{i,j = 1}^{N} a_{i,j} \partial_{y_i y_j} u  + 
    \sum_{i,j = 1}^{N} b_{i,j} y_j \partial_{y_i} u + \partial_{t} u, \qquad (y,t) \in \mathbb{R}^{N+1},
\end{equation}
where $\widetilde A=(a_{i,j})_{i,j=1,\ldots,N}$ and $B=(b_{i,j})_{i,j=1,\ldots,N}$ are constant matrices,  with 
$\widetilde A$ symmetric and non-negative. We can choose, as it is not restrictive, a basis of $ \mathbb{R}^N$ such 
that $\widetilde A$ takes the following form
\begin{equation*} 
    \widetilde A = \begin{pmatrix}  A & 0 \\ 0 & 0 \end{pmatrix},
\end{equation*}
with the constant matrix $A=(a_{i,j})_{i,j=1,\ldots,m_0}$ strictly positive. 

Clearly, as $m_0 < N$ the operator $\K$ is strongly degenerate. The regularity properties of $\K$ depend on its first 
order part
\begin{equation}\label{eq-Y}
  Y = 
  \langle B y, \nabla \rangle + \p_{t}.
\end{equation}
In order to clarify this assertion, we introduce some further notation. Let $C = (c_{i,j})_{i,j=1,\ldots,N}$ denote the 
\emph{square root of $\widetilde A$}, that is the unique positive symmetric matrix such that $C^2 = \widetilde A$. Then 
$\K$ can be written as 
\begin{equation}\label{eq-KK}
  \K = \sum_{j=1}^{m_0} X_{i}^2+ Y,
\end{equation}
with
\begin{equation}\label{eq-VFK}
  X_{i}=\sum_{j=1}^{N}c_{ij}\p_{y_{j}}, \quad i=1, \dots, m_{0}.
\end{equation}
With the above notation, the following statements are equivalent:
\begin{description}
    \item [{$(H_1)$}] there exists a basis of $\rn$ such that $B$ has the form
\begin{equation}\label{eq-B}
\left(
\begin{array}{ccccc}
              \ast & \ast & \ldots & \ast & \ast \\
              B_1 & \ast & \ldots & \ast & \ast \\
              0 & B_2 & \ldots & \ast & \ast  \\
              \vdots & \vdots & \ddots & \vdots & \vdots \\
              0 & 0 & \ldots & B_\kappa & \ast 
              \end{array} \right),
\end{equation}
where $B_j$ is a matrix $m_{j} \times m_{j-1}$ of rank $m_j$, with
$$ m_0 \geq m_1 \geq \ldots \geq m_\kappa \geq 1, \hspace{1.5cm} 
m_0 + m_1 + \ldots + m_\kappa = N,
$$
while $\ast$ are constant and arbitrary blocks;
    \item [{$(H_2)$}] H\"{o}rmarder's condition:
\begin{equation}\label{eq-rank}
    \text{rank\ Lie}\left(X_{1},\dots,X_{m_{0}},Y\right)=N+1,
    \quad \text{at every point of}\ \rNN;
\end{equation}
    \item [{$(H_3)$}] Kalman's controllability condition:
\begin{equation} \label{eq-Kal}
    \text{rank} \big( C, B C, \dots, B^{N-1} C \big) = N,
\end{equation}
(see \cite{LeeMarkus}, Theorem 5, p. 81).
\end{description}
For the equivalence of the above conditions we refer to \cite{LP94}. In the sequel, we assume that the basis of 
$\rn$ is as in $(H_1)$.

\medskip

We note that the regularity properties of the differential operator $\K$ are related to some differential 
properties of the vector fields $X_{1},\dots,X_{m_{0}}, Y$. As we said in the Introduction, this fact was the 
starting point of the regularity theory for degenerate H\"ormander operators developed in 
 \cite{Hormander, Folland, RothschildStein,NagelSteinWainger}. For this reason we will recall some basic facts about 
the 
Lie groups related to $\K$.

It is known that every operator $\K$ is invariant with respect to a \emph{non-Euclidean} translation defined as 
follows. For every $(y_0, t_0), (y,t) \in \rNN$ we set
\begin{equation}\label{eq-LieGR}
  (y_0, t_0) \circ  (y,t) :=(y +\exp(t B) y_0,t+t_0).
\end{equation}
If $u$ is a solution of the equation $ \K u = f$ in some open set $\O \subset \rNN$, then the 
function $v(y,t) := u \left( (y_0, t_0) \circ (y,t) \right)$ is solution to $\K v = g$, where $g(y,t) := f \left( (y_0, 
t_0) \circ (y,t) \right)$ in the set $\big\{ (y,t) \in \rNN \mid  (y_0, t_0) \circ (y,t) \in \O \big\}$. It is known 
that $\rNN$ with the operation ``$\circ$'' is a non commutative group, with identity $(0,0)$. The inverse of $(y,t)$ is 
\begin{equation}\label{eq-inverse}
  (y, t)^{-1}  =(- \exp(- t B) y,-t).
\end{equation}
Moreover, if (and only if) all the $*$-block in \eqref{eq-B} are null, then $\K$ is homogeneous of degree two with 
respect to the family of the following dilatations,
\begin{equation}\label{eq-dilGR}
   d_r :=  \text{\rm diag}\left(r I_{m_{0}},r^{3}  I_{m_{1}},\dots,r^{2\kappa+1} I_{m_{\kappa}},r^{2}\right),
\end{equation}
($I_{m_{j}}$ denotes the $m_{j}\times m_{j}$ identity matrix). In this case the following \emph{distributive property} 
of the dilation holds
\begin{equation*} 
  \big( d_r (y_0, t_0) \big) \circ \big( d_r (y,t) \big)  = d_r \big(  (y_0, t_0) \circ (y,t) \big), 
\end{equation*}
for every $(y_0, t_0), (y,t) \in \rNN$ and for every $r>0$. In literature the structure
\begin{equation} \label{eq-LieGroup}
  \mathbb{L} := \big( \rNN, \circ , (d_r)_{r>0} \big),
\end{equation}
is usually referred to as \emph{homogeneous Lie group}.
We quote \cite{LP94} for the main properties of the Lie group $\mathbb{L}$ defined by \eqref{eq-LieGR}, 
\eqref{eq-dilGR}.

\medskip

We now introduce some basic notions of the Control Theory in order to describe the set where the Harnack inequality 
holds for the non-negative solutions of $\L u = f$ . As noticed above, the link between the Regularity Theory for 
linear PDEs and the Control Theory is not surprising, as the hypoellipticity of $\K$ is equivalent to the 
controllability condition $(H_3)$. The first notion we need is the $\L$-{\sc admissible curve}, the second one is that  
of {\sc attainable set}. For the precise statement of them we first consider the operator $\K$ in \eqref{eq-Kolmo-RN} 
and we recall the relevant notation \eqref{eq-KK}. We say that a curve $ \g: [0,T] \rightarrow \rNN$ is $\K$-{\sc 
admissible} if:
\begin{itemize}
	\item it is absolutely continuous;
	\item $\dot{\g}(s) = \sum\limits_{k=1}^{m_0} \omega_k(s) X_k(\g(s)) + Y(\gamma(s))$ 
			a.e. in $[0,T]$, with $\o_1, \o_2, \dots, \o_{m_0} \in L^{1}[0,T]$.
\end{itemize}
Moreover we say that $\g$ steers $ (y_{0}, t_{0})$ to $(y,t)$, for $t_0 > t$, if $\g(0) =  (y_{0}, t_{0})$ and 
$\g(T) = (y, t)$.
Note that $t(s) = t_0 - s$, then $T= t_0- t$ and $t_0 > t$. We denote by $\ASY ( \O )$ the following set:
\begin{equation*}
	\ASY  ( \O ) = 
	\begin{Bmatrix}
	(y,t) \in \O \mid \hspace{1mm} \text{\rm there exists a} \ \K - \text{\rm admissible curve}
	\ \g : [0,T] \rightarrow \O \hspace{1mm} \\ 
	\hfill \text{\rm such that} \ 	\g(0) = (y_{0}, t_{0}) \hspace{1mm} {\rm and}  \hspace{1mm} \g(T) = (y,t).
	\end{Bmatrix}.
\end{equation*}
We will refer to $\ASY (\O)$ as {\sc attainable set}. 

\medskip

In the sequel of this Section we focus on the equation \eqref{eq-Kolmo}, which writes in the  form \eqref{eq-Kolmo-RN} 
if we choose $N= 2n, y = (v,x)$,
\begin{equation*} 
    A = I_n, \quad \text{and} \quad B = \begin{pmatrix}  0_n & 0_n \\ I_n & 0_n\end{pmatrix}.
\end{equation*}
Here $0_n$ and $I_n$ denote the zero and the identity $n \times n$ matrices, respectively. In particular, $\L_0$ 
satisfies the condition $(H_1)$ and is invariant with respect to a dilation of the form \eqref{eq-dilGR}. Moreover, if 
we 
identify any vector field $X =\sum_{j=1}^{2n}c_{j} \p_{y_{j}}$ with the vector $\sum_{j=1}^{2n}c_{j}e_{j}$, being $e_j$ 
the the $j^{\rm th}$ vector of the canonical basis of $\rn$, then $\L_0$ writes in the form \eqref{eq-KK} provided that 
we set
\begin{equation*}
  Y = \langle v, \nabla_{x} \rangle + \partial_{t} \sim 
\begin{pmatrix}
0\\
\vdots\\
0 \\
v_{1}\\
\vdots \\
v_{n}\\
1
\end{pmatrix}, 
\qquad 
X_j = \partial_{v_{j}} 
\sim e_{j} = 
\begin{pmatrix}
0\\
\vdots\\
1 \denoterow{1} \\
0\\
\vdots\\
0
\end{pmatrix}
\hspace{15mm} {\rm for} \hspace{1mm} j=1, \ldots, n.
\end{equation*}

In this setting the Lie group $\mathbb{L}$ in \eqref{eq-LieGroup} is defined in terms of the following \emph{Galilean 
change of coordinate} in the Phase Space,
\begin{equation} \label{eq-Gal-translation}
  (v_{0},x_{0},t_{0}) \circ (v,x,t) := (v + v_0, x_0 + x + t v_0, t_0 + t), 
  \qquad (v_{0},x_{0},t_{0}), (v,x,t) \in \rnn.
\end{equation}
Moreover, $\L_0$ is invariant with respect to the following 
\begin{equation} \label{eq-Gal-dilation}
  d_{r} (v,x,t) = (rv, r^{3}x, r^{2}t), \qquad (v,x,t) \in \rnn, r>0.
\end{equation}
In the sequel we will denote by $\mathbb{G}$ the group defined in terms of \eqref{eq-Gal-translation} and  
\eqref{eq-Gal-dilation}
\begin{equation} \label{eq-GalGroup}
  \mathbb{G} := \big( \rnn, \circ , (d_r)_{r>0} \big),
\end{equation}

\medskip

When we consider the  operator $\K = \L_0$, the $\K$-admissible curves can be easily described. Indeed, if we denote 
\begin{equation*} 
  \g(s) = (v(s), x(s), t(s)), \quad s \in [0,T],
\end{equation*}
then the problem
\begin{equation*}
  \dot{\g}(s) = \sum_{k=1}^{m_0} \omega_k(s) X_k(\g(s)) + Y(\gamma(s)), \quad \g(0) = (y, t), \ \g(T) = (\y, \t),
\end{equation*}
becomes
\begin{equation} \label{eq-L-admissible}
  \dot v(s) = \omega(s), \quad \dot x(s) = v(s), \quad \dot t(s) = -1,
\end{equation}
and its solution is 
\begin{equation*} 
  v(s) = v_{0} + \int_{0}^s \omega(\tau) d\, \tau, \quad x(s) = x_{0} + \int_{0}^sv(\tau) d\, \tau, \quad t(s) = t_0-s,
\end{equation*}
The controllability condition $(H_3)$ guarantees that, for every $(v,x,t) \in \rnn$, with $t<t_0$, there is at least a 
\emph{control} $\omega =(\o_1, \o_2, \dots, \o_{n}) \in (\L^1 [0,T])^n$ such that the solution to 
\eqref{eq-L-admissible} satisfies $(v(T),x(T),t(T)) = (v,x,t)$. In the sequel we will use the following notation

\begin{definition} \label{curva-amm}
A curve $ \g = (v,x,t): [0,T] \rightarrow \rnn$ is said to be $\L$-{\sc admissible} if it is absolutely continuous, and 
solves the equation \eqref{eq-L-admissible} for almost every $s \in [0,T]$, with $\o_1, \o_2, \dots, \o_{n} \in 
L^{1}[0,T]$. 
Moreover we say that $\g$ steers $(v_0,x_0,t_0)$ to $(v,x,t)$, with $t_0 > t$, if $\g(0) = (v_0,x_0,t_0)$ and $\g(T) = 
(v,x,t)$.
\end{definition}

\medskip 

\begin{center}
\begin{pspicture}(5,5)
	\label{gamma}
    \pcline[linecolor=dgrey](2,1)(2,4.8)
    \pcline[linecolor=dgrey](2,4.8)(1.95,4.6)%
    \pcline[linecolor=dgrey](2,4.8)(2.05,4.6)%
    \uput[0](2,4.8){$t$}
    \pcline[linecolor=dgrey](1.7,2.1)(5,1)%
    \pcline[linecolor=dgrey](5,1)(4.8,1)%
    \pcline[linecolor=dgrey](5,1)(4.9,1.1)%
    \uput[0](5,1){$v$}
    \pcline[linecolor=dgrey](2.2,2.1)(0,1)%
    \pcline[linecolor=dgrey](0,1)(0.1,1.1)%
    \pcline[linecolor=dgrey](0,1)(0.2,1)%
    \uput[180](0,1){$x$}
    \psbezier[linecolor=dgrey, linewidth=1pt]
      (0,2.5)(2,5)(4,1)(5,5)
    \uput[0](-3.25,2.5){$\gamma(T) = (v,  x, t)$}
    \uput[0](5,5){$\gamma(0) = (v_0,  x_0, t_0)$}
    \psset{linecolor=dgrey}
    \qdisk(0,2.5){2pt}
    \psset{linecolor=dgrey}
    \qdisk(5,5){2pt}
    \psline[linecolor=dgrey]{*->}(4,3.2)(5.2,2.9)
    \uput[90](5.5,2.9){$X = \partial_v$}
    \psline[linecolor=dgrey]{*->}(4,3.2)(2.8,3.5)
    \uput[90](3.4,3.3){$- X$}
    \psline[linecolor=dgrey]{*->}(4,3.2)(2.5,2.15)
    \uput[135](2.9,2.3){$\dot \gamma(s)$}
    \psline[linecolor=dgrey]{*->}(4,3.2)(3.3,1.9)
    \uput[0](3.5,2.1){$Y = v \partial_x - \partial_t$}
    \pcline[linecolor=dgrey, linestyle=dashed](2.5,2.15)(3.3,1.9)
    \pcline[linecolor=dgrey, linestyle=dashed](2.5,2.15)(3.2,3.4)
  \end{pspicture}
\end{center}
\begin{center}
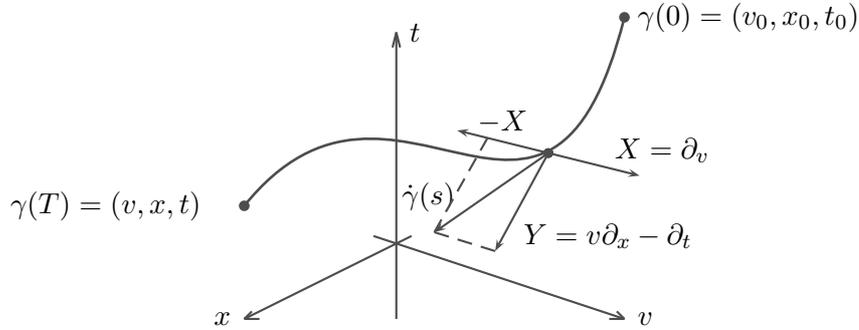
  {\scriptsize \sc Fig. 3 - An $\L-$admissible curve steering $(v_{0}, x_{0}, t_{0})$ to $(v, x, t)$.}
\end{center}

\medskip

\begin{definition} \label{prop-set}
Let $\O$ be any open subset of $\rnn$, and let $(v_0,x_0,t_0) \in \O$. We denote by $\AS ( \O )$ the following 
set:
\begin{equation*}
	\AS  ( \O ) = 
	\begin{Bmatrix}
	(v,x,t) \in \O \mid \hspace{1mm} \text{\rm there exists an} \ \L - \text{\rm admissible curve}
	\ \g : [0,T] \rightarrow \O \hspace{1mm} \\ 
	\hfill \text{\rm such that} \ \g(0) = (v_{0}, x_{0}, t_{0}) \hspace{1mm} {\rm and}  
	\hspace{1mm} \g(T) = (v,x,t).
	\end{Bmatrix}.
\end{equation*}
We will refer to $\AS (\O)$ as {\sc attainable set}. We will use the notation $\AS = \AS ( \O )$ whenever there is no 
ambiguity on the choice of the set $\O$ . 
\end{definition}

\section{Harnack inequalities}
\setcounter{equation}{0}
In this Section we recall the Harnack inequality for equation $\L u = f$ due to Golse, Imbert, Mouhot and Vasseur (see 
Theorem 3 in  \cite{4authors}), then we prove some preliminary results useful for the proof of our Theorem 
\ref{mainthe}. 

Let $Q = ]-1,1[^{2n} \times ]-1,0[$ be the \emph{unit box} introduced in \eqref{eq-unitbox}. Based on the dilation 
\eqref{eq-Gal-dilation} and on the Galilean translation \eqref{eq-Gal-translation}, for every positive $r$ and for 
every $(v_{0}, x_{0}, t_{0})$ we define the sets
\begin{equation*} 
\begin{split}
  & Q_{r} := d_r Q = \big\{d_r (v,x,t) \mid (v,x,t) \in Q \big\},\\
  & Q_{r}(v_{0}, x_{0}, t_{0}) := (v_{0}, x_{0}, t_{0}) \circ d_r Q = \\
  & \qquad \qquad \qquad \big\{ (v_{0}, x_{0}, t_{0}) \circ d_r (v,x,t) \mid (v,x,t) \in Q \big\}.
\end{split}
\end{equation*}
A direct computation shows that 
\begin{equation*} 
\begin{split}
	Q_{r} \hspace{1mm} =  \hspace{1mm} ]-r,r[^n & \times ]-r^{3},r^{3}[^n \times ]-r^{2},0[, \\
	Q_{r}(v_{0}, x_{0}, t_{0}) = & 
	\Big\{ (v,x,t) \in \rnn \mid |(v - v_{0})_j| < r,  \\
	& |(x- x_{0} - (t - t_{0}) v_{0})_j | < r^{3}, j=1, \dots, n,  t_{0}- r^2  < t < t_{0} \Big\}.
\end{split}
\end{equation*}
With the above notation, the following result holds.
\begin{theorem} \label{harnack}{\rm (Theorem 2 in  \cite{4authors})} There exist three constants $M > 1, R >0, \D >0$, 
with $0 < R^{2} < \D < \D + R^{2} < 1$, such that  
\begin{equation*} 
		\sup_{Q^{-}} u \le M (\inf_{Q^{+}} u + \left\| f \right\|_{L^{\infty} (Q))})
\end{equation*}
for every non-negative weak solution $u$ to the equation $\L u = f$ on $Q$, with $f \in L^\infty(Q)$. The constants $M, 
R$ and $\D$ only depend on the dimension $n$ and on the ellipticity constants $\l$ and $\LL$. Moreover $Q^{+}, Q^{-}$ 
are defined as follows
\begin{equation*} 
	Q^{+} = Q_{R} \hspace{2mm} {\rm with}
					\hspace{2mm} 0 < R^{2} < \D < \D + R^{2} < 1, \hspace{8mm}
	Q^{-} = Q_{R}(0, 0, -\D).
\end{equation*}
\end{theorem}

As Golse, Imbert, Mouhot and Vasseur notice in Remark 4 in \cite{4authors}, ``using the transformation 
\eqref{eq-Gal-translation}, we get a Harnack inequality for cylinders centered at an arbitrary point $(v_0, 
x_0,t_0)$''. We next give a precise meaning to this assertion and we improve it by also using the dilation 
\eqref{eq-Gal-dilation}.

\begin{theorem} \label{harnack-inv} Let $(v_{0}, x_{0}, t_{0})$ be any point of $\rnn$ and let $r$ be a positive number.
There exist three constants $M > 1, R >0, \D >0$, with $0 < R^{2} < \D < \D + R^{2} < 1$, such that  
\begin{equation*} 
		\sup_{Q^{-}_{r}(v_{0}, x_{0}, t_{0})} u \le M (\inf_{Q^{+}_{r}(v_{0}, x_{0}, t_{0})} u + \left\| f 
\right\|_{L^{\infty} (Q_{r}(v_{0}, x_{0}, t_{0})})
\end{equation*}
for every non-negative weak solution $u$ to the equation $\L u = f$ on $Q_{r}(v_{0}, x_{0}, t_{0})$, with $f \in 
L^\infty(Q_{r}(v_{0}, x_{0}, t_{0}))$. The constants $M, R$ and $\D$ only depend on the dimension $n$ and on 
the ellipticity constants $\l$ and $\LL$. Moreover $Q^{+}_{r}(v_{0}, x_{0}, t_{0}), ^{-}Q_{r}(v_{0}, x_{0}, t_{0})$ are 
defined as follows
\begin{equation*} 
	Q^{+}_{r}(v_{0}, x_{0}, t_{0}) = (v_{0}, x_{0}, t_{0}) \circ  d_r Q^{+}, \qquad
	Q^{-}_{r}(v_{0}, x_{0}, t_{0}) = (v_{0}, x_{0}, t_{0}) \circ  d_r Q^{-}.
\end{equation*}
\end{theorem}

\begin{proof} We rely on the invariance of the operator $\L_0$ with respect to the group \eqref{eq-GalGroup}. If $u$ is 
a non-negative solution to $\L u = f$ in $Q_{r}(v_{0}, x_{0}, t_{0})$, then the function $\widetilde u (v,x,t) := u 
\left( d_{1/r} \big((v_0, x_0, t_0)^{-1} \circ (v,x,t) \big) \right)$ is a solution in the unit box $Q$ to the  
following equation 
\begin{equation*} 
\widetilde \L \widetilde u =: \partial_t \widetilde u  + \langle v, \nabla_x \widetilde u \rangle - 
{\rm div}_v (\widetilde A \nabla_v \widetilde u) - \langle \widetilde b, \nabla_v \widetilde u \rangle = \widetilde f.
\end{equation*}
Here $\widetilde A (v,x,t) := A \left( d_{1/r} \big( (v_0, x_0, t_0)^{-1} \circ (v,x,t) \big) \right)$, 
$\widetilde b (v,x,t) := b \left( d_{1/r} \big( (v_0, x_0, t_0)^{-1} \circ (v,x,t) \big) \right)$ and $\widetilde f 
(v,x,t) := f \left( d_{1/r} \big( (v_0, x_0, t_0)^{-1} \circ (v,x,t) \big) \right)$. Moreover $(v_0, x_0, 
t_0)^{-1}$ is defined in \eqref{eq-inverse}. Even though $\widetilde \L$ does not agree with $\L$, it satisfies the 
hypotheses of Theorem \ref{harnack} with the same structural constants $n$,  $\l$ and $\LL$. We then apply 
Theorem \ref{harnack} to the function $\widetilde u$ and we plainly obtain our claim for $u$.  
\end{proof}

An useful tool in the proof of our main theorem is the following lemma (Lemma 2.2 in \cite{BP07}). 
To give here its statement we introduce a further notation. We choose any $S \in ]0,R[$ and we set 
\begin{equation*}
  K^- = [ - S, S ]^n \times \left[ - S^3, S^3 \right]^n \times \big\{ - \left( \Delta + R^2/2 \right)  \big\}.
\end{equation*}
Moreover, for every $(v,x,t) \in \rnn$ and $r>0$ we let
\begin{equation*}
  K^-_r(v,x,t) = (v,x,t) \circ d_r (K^-).
\end{equation*}
We have that 
\begin{lemma} \label{lemma2.2}
Let $\g: [0,T] \rightarrow \rnn$ be an $\L-$admissible path and let $a, b$ be two constants s.t. $0 \le a < b \le T$. 
Then there exists a positive constant $h$, only depending on $\L$, such that
\begin{equation*}
		\int_{a}^{b} |\o (\t)|^{2} \d\t \le h \hspace{4mm} \implies 
		\hspace{4mm} \g(b) \in K^{-}_{r}(\g(a)), \hspace{2mm} {\rm with} \hspace{2mm}
		r = \sqrt{ \frac{b - a}{(\D + 1/2)}}.
\end{equation*}
\end{lemma}

\begin{remark} \label{rem-l22}
  Note that $K^-_r(v,x,t)$ is a compact subset of $Q^-_r(v,x,t)$ for every $(v,x,t) \in \rnn$ and for any $r>0$.  As 
a consequence of Lemma \ref{lemma2.2}, $K^{-}_{r}(\g(a))$ is an open neighborhood of $\g(b)$.
\end{remark}

\section{Proof of the main results} 
\setcounter{equation}{0}
An useful notion in the proof of our main result is that of {\sc Harnack chain}.
\begin{definition} \label{def-chain}
	We say that $\{ z_{0}, \ldots, z_{k} \} \subseteq \O$ is a Harnack chain connecting 
	$z_{0}$ to $z_{k}$ if there exist $k$ positive constants $C_{1}, \ldots, C_{k}$ 
	such that
	\[ 
	u(z_{j}) \le C_{j} u(z_{j-1})) \hspace{5mm} j = 1, \ldots, k
	\]
	for every non-negative solution $u$ of $\L u = f$ in $\O$.
\end{definition}
Our first result of this Section is a local version of Theorem \ref{mainthe}.
\begin{proposition} \label{proposition}
For every $(v, x, t) \in {\rm int} \big( \A_{(v_{0}, x_{0}, t_{0})} \big)$ there exist an open neighborhood 
$U_{(v,x,t)}$ of $(v, x, t)$ and a positive constant $C_{(v, x, t)}$ such that
\begin{equation*}
		\sup_{U_{(v,x,t)}} u \le C_{(v, x, t)} \hspace{1mm} \left( u(v_{0},x_{0},t_{0}) 
		\hspace{1mm} + \hspace{1mm} \left\| f \right\|_{L^{\infty}(\O)} \right),
\end{equation*}
for every non-negative solution to $\L u = f$, with $f \in L^{\infty}(\O)$.
\end{proposition}

\begin{proof}
Let $(v,x,t)$ be any point of ${\rm int} \big( \A_{(v_{0}, x_{0}, t_{0})} \big)$. We plan to prove our claim by 
constructing a finite Harnack chain connecting $(v, x, t)$ to $(v_{0}, x_{0}, t_{0})$. Because of the very definition 
of $\A_{(v_{0}, x_{0}, t_{0})}$, there exists a $\L-$admissible curve $\g : [0, T] \rightarrow \O$ steering $(v_{0}, 
x_{0}, t_{0})$ to $(v,x,t)$. Our Harnack chain will be a finite subset of $\g([0,T])$.


In order to construct our Harnack chain, we introduce a further notation. Let $\widetilde Q := ]-1,1[^{2n+1}$ and note 
that it is an open neighborhood of the origin of $\rnn$. Because of the continuity of the Galilean change of variable 
``$\circ$'' and of the dilation $\left( d_r \right)_{r>0}$, for every $(v',x',t') \in \rnn$, the family 
\begin{equation}
  \left( \widetilde Q_r(v',x',t') \right)_{r>0}, \qquad \widetilde Q_r(v',x',t') := (v',x',t') \circ  d_r \widetilde Q,
\end{equation}
is a  neighborhood basis of the point $(v',x',t')$. Then, again because of the continuity of ``$\circ$'' and $\left( 
d_r \right)_{r>0}$, for every $s \in [0,T]$ there exists a positive $r$ such that $\widetilde Q_{r}(\g(s)) \subseteq 
\O$. Thus we can define
\begin{equation} \label{rs}
		r(s) \hspace{1mm} := \hspace{1mm} \sup \left\{ r > 0 \hspace{1mm} : \hspace{1mm}
		\widetilde Q_{r}(\g(s)) \subseteq \O \right\}.
\end{equation}
Note that the function \eqref{rs} is continuous, then it is well defined the positive number 
\begin{equation} \label{r}
	r_0 \hspace{1mm} := \hspace{1mm} \min_{s \in [0,T]} r(s). 
\end{equation}
As $Q_{r}(\g(s)) \subset \widetilde Q_{r}(\g(s))$, we conclude that 
\begin{equation} \label{Qr}
	Q_{r}(\g(s)) \subseteq \O \quad \text{for every} \ s \in [0,T] \quad \text{and} \ r \in ]0, r_0]. 
\end{equation}
On the other hand, we notice that the function 
\begin{equation}
	I(s) \hspace{1mm} := \hspace{1mm} \int_{0}^{s} | \o (\t)|^{2} dt, 
\end{equation}
is (uniformly) continuous in $[0,T]$, then there exists a positive $\delta_0$ such that $\delta_0 \le (\D 
+ R^{2}/2) r_0$ and that 
\begin{equation} \label{delta}
	\int_{a}^{b} | \o (\t)|^{2} dt \le h \qquad \text{for every} \ 
	a,b \in [0,T], \quad \text{such that} \  0 < a-b \le \d_0, 
\end{equation}
where $h$ is constant appearing in Lemma \ref{lemma2.2}. 

We are now ready to construct our Harnack chain. Let $k$ be the unique positive integer such that $(k-1) \d_0 < T$, and 
$k \d_0 \ge T$. 
We define $\{ s_{j} \}_{j \in \{ 0, 1, \ldots, k \}} \in [0,T]$ as follows: $s_j = j \d_0$ for $j=0,1, \dots, k-1$, 
and $s_k =T$.
As noticed before, the equation \eqref{delta} allows us to apply Lemma \ref{lemma2.2}. We then obtain
\begin{equation} \label{propsuccessione}
	\g(s_{j+1}) \in Q^{-}_{{r_0}} (\g(s_{j}))  \hspace{5mm} j=0, \ldots, k-2, \qquad \g(s_k) \in Q^{-}_{{r_1}} 
(\g(s_{k-1})),
\end{equation}
for some $r_1 \in ]0, r_0]$. We next show that $\left( \g(s_{j}) \right)_{j=0,1, \dots,k}$ is a Harnack chain and we 
conclude the proof.  We proceed by induction. 
For every $j = 1, \ldots, k-2$ we have that $\g (s_{j+1}) \in Q^{-}_{{r_0}} (\g(s_{j}))$. From \eqref{Qr} we know that 
$Q_{{r_0}} (\g(s_{j})) \subseteq \Omega$, then we apply Theorem \ref{harnack} and we find

\begin{align*}
	u(\g(s_{j+1})) \le  
	\sup_{Q^{-}_{{r_0}}(\g (s_{j}))} u \le 
	\ M 
	\Big( \inf_{Q^{+}_{{r_0}}(\g (s_{j}))} u + 
	& \ \left\| f \right\|_{L^{\infty} (Q (\g (s_{j})))} \Big) \\
	& \ \le M
	\Big( u(\g (s_{j})) 
	+ \left\| f \right\|_{L^{\infty} (\Omega)} \Big).
\end{align*}
Here we rely on the fact that $u$ is a continuous function. As a consequence we obtain
\begin{align*}
u(\g (s_{k-1})) &\le M ( u(\g(s_{k-2})) + \left\| f \right\|_{L^{\infty} (\O) }) \\
	      &\le M ( M ( u(\g(s_{k-3})) + \left\| f \right\|_{L^{\infty} (\O)}) + \left\| f \right\|_{L^{\infty} (\O) 
}) \\
	  	  &\setbox0\hbox{=}\mathrel{\makebox[\wd0]{\hfil\vdots\hfil}} \\
	      &\le M^{k-1} u(\g(0)) + \left\| f \right\|_{L^{\infty} (\O) } 
	      \sum_{i=1}^{k-1} M^{i}.
\end{align*}
We eventually apply Theorem \ref{harnack} to the set $Q_{{r_1}} (\g(s_{k-1})) \subseteq \Omega$ and we obtain
\begin{equation*}
		\sup_{U_{(v,x,t)}} u \le C_{(v, x, t)} \hspace{1mm} \left( u(v_{0},x_{0},t_{0}) 
		\hspace{1mm} + \hspace{1mm} \left\| f \right\|_{L^{\infty}(\O)} \right),
\end{equation*}
where $C_{(v, x, t)} = \sum_{i=1}^{k} M^{i}$ and $U_{(v,x,t)} = Q^{-}_{{r_1}} (\g(s_{k-1}))$. As we noticed in 
Remark \ref{rem-l22}, $Q^{-}_{{r_1}} (\g(s_{k-1}))$ is an open neighborhood of $\g(T)$. This concludes the Proof of 
Proposition \ref{proposition}. \end{proof}

\medskip

\noindent {\sc Proof of Theorem \ref{mainthe}}. 
Let  $K$ be any compact subset of ${\rm int} \left( \A_{(v_{0}, x_{0}, t_{0})} \right)$.  For every $(v,x,t) \in K$ we 
consider the open set $U_{(v,x,t)}$. Clearly we have
\begin{equation*}
	K \hspace{1mm} \subseteq \hspace{1mm} \bigcup_{(v,x,t) \in K} \hspace{1mm} U_{(v,x,t)}.
\end{equation*}
Because of its compactness, there exists a finite covering of $K$
\begin{equation*}
	K \hspace{1mm} \subseteq \hspace{1mm} \bigcup_{j=1, \ldots, m_{K}} \hspace{1mm} 		
					U_{(v_{j},x_{j},t_{j})},
\end{equation*}
and Proposition \ref{proposition} yields
we 
\begin{equation*}
	\sup_{U_{(v_{j},x_{j},t_{j})}} u \le C_{(v_{j},x_{j},t_{j})} \hspace{1mm} \left( 
	u(v_{0},x_{0},t_{0}) \hspace{1mm} + \hspace{1mm} \left\| f \right\|_{L^{\infty}
	(\O)} \right) \hspace{8mm} j=1, \ldots, m_{K}.
\end{equation*} 
This concludes the proof of Theorem \ref{mainthe}, if we choose
\begin{equation*}
	C_{K} \hspace{1mm} = \hspace{1mm} \max_{j=1, \ldots, m_{K}} C_{(v_{j}, x_{j}, t_{j})}.
\end{equation*}
\hfill $\square$

\medskip 

\noindent {\sc Proof of Theorem \ref{maincor}}. If $u$ is a non-negative solution to $\L u = 0$ in $\Omega$ and $K$ 
is a compact subset of $\A$, then $\sup_{K} u \le C_{K} u(v_{0}, x_{0}, t_{0})$. If moreover $u(v_{0}, x_{0}, t_{0}) = 
0$, we have $u(v,x,t) = 0$ for every $(v,x,t) \in K$ and, thus, $u(v,x,t) = 0$ for every $(v,x,t) \in \AS$. The 
conclusion of the proof then follows from the continuity of $u$. \hfill $\square$

\end{document}